\documentclass[12pt]{amsart}
\usepackage{geometry}
\usepackage{amsthm, amssymb, thm-restate, mathtools}
\usepackage{esint}
\usepackage{color}
\usepackage{array}
\usepackage{booktabs}
\usepackage{multirow}
\usepackage{graphicx}
\usepackage{tikz}
\usetikzlibrary{arrows,shapes,positioning,shadows,trees}

\usepackage{comment}

\makeatletter

\numberwithin{equation}{section}
\numberwithin{figure}{section}

\makeatother

\newcommand{\be}{\begin{equation}}
\newcommand{\ee}{\end{equation}}
\newcommand{\NN}{\mathbb{N}}
\newcommand{\RR}{\mathbb{R}}

\newtheorem{thm}{Theorem}

\newtheorem{rmk}[thm]{Remark}

\newtheorem{lem}[thm]{Lemma}
\newtheorem*{lem*}{Lemma}
\newtheorem*{thm*}{Theorem}
\newtheorem{cor}[thm]{Corollary}
\newtheorem*{cor*}{Corollary}

\begin{document}

\title[Manifold embeddings by heat kernels of connection Laplacian]{Manifold embeddings by heat kernels of connection Laplacian} 
\author{Chen-Yun Lin}
\address{Chen-Yun Lin\\
Department of Mathematics\\
Lehman College, CUNY}
\email{Chenyun.Lin@lehman.cuny.edu}



\maketitle

\begin{abstract}
We show that any closed $n$-dimensional manifold $(M,g)$ can be embedded by a map constructed using the heat kernels of the connection Laplacian as well as a maps constructed using truncated heat kernel at a certain time $t$ from a $\delta$-net $\{q_i\}_{i=1}^{N_0}$ via a rescaling trick. Both the time $t$ and $N_0$ are bounded in terms of the dimension, bounds on the Ricci curvature and its derivative, the injectivity radius, and the volume. Moreover, both maps can be made arbitrarily close to an isometry. 
\end{abstract}

\newpage

\section{Introduction}

Data collected for the purpose of machine learning is often in a high-dimensional space, but yet is believed to satisfy certain low dimensional structure, that is, the collected dataset can be well approximated by a low dimensional manifold  sitting inside a high dimensional Euclidean space. See \cite{Belkin_Niyogi:2003, Coifman_Lafon:2006, hein2007manifold,fefferman2018fitting} for a far from complete list of available literature.
How to analyze a dataset under this assumption is generally called the manifold learning problem.
One particular goal is to recover the nonlinear low dimensional structure of the manifold and to reduce the dimensionality of the space where the dataset lies inside. 
Mathematically, this problem is formulated as asking if it is possible to have an embedding to put the manifold (hence the dataset)
into a finite dimensional Euclidean space that is bi-Lipschitz, even isometric.
Although the embedding problem was first positively answered by Whitney \cite{Whitney:1944}, and the isometrically embedding problem was
first solved by Nash \cite{Nash:1956},  the approaches  are not canonical and are not essentially feasible for data analysis. 
In Berard-Besson-Gallot’s breakthrough paper \cite{Berard_Besson_Gallot:1994}, the spectral embedding idea is explored.
 They show that  a manifold can be embedded into the sequence space $\ell^2$  with all  eigenfunctions. 

The spectral embedding idea is directly related to many manifold learning algorithms, like
eigenmaps \cite{Belkin_Niyogi:2003}, local linear embeddings \cite{Roweis_Saul:2000} and diffusion maps \cite{Coifman_Lafon:2006}.
There is also a rich theoretical literature describing how the graph Laplacian converges to the Laplace-Beltrami operator; for example,
the pointwise convergence wast discussed in  \cite{Belkin_Niyogi:2003, Coifman_Lafon:2006, Tenenbaum_deSilva_Langford:2000}  and the spectral convergence was discussed in \cite{Belkin_Niyogi:2007, Singer_Wu:2016, Eldridge_Belkin_Wang:2017, Trillos_et_al:2018, Calder:2019,dunson2019diffusion}, where the $L^\infty$ spectral convergence with rate was recently explored in \cite{dunson2019diffusion}.
Note that numerically we are able to obtain only finite eigenfunctions
and eigenvalues in practice since only finite number of points are available.
Thus, the next natural question  is whether we can embed the manifold by finite eigenfunctions and eigenvalues.
This question was positively answered separately by Bates \cite{Bates:2014} and Portegies \cite{Portegies:2015}  for the Laplace-Beltrami operator; that is, one is able to construct bi-Lipschitz embeddings of the manifold with finite eigenfunctions and eigenvalues of the Laplace–Beltrami operator.
In Portegies \cite{Portegies:2015}, it is further shown that the embedding is almost isometric with a prescribed error bound. 
Moreover,  Jones et al. \cite{Jones_Maggioni_Schul:2008} observed that in terms of analysis it may be easier to think of embedding with heat kernels. 

The spectral embedding mentioned above depends on the Laplace-Beltrami operator. The
vector diffusion map (VDM) \cite{Singer_Wu:2012, Singer_Wu:2016}, on the other hand, depends on the connection Laplacian
associated with a possibly nontrivial bundle structure. 
The VDM is originally motivated by studying
the cryo-electron microscope problem, in particular the class averaging algorithm \cite{Hadani_Singer:2011b, Singer_Zhao_Shkolnisky_Hadani:2011, Zhao_Singer:2014}. Its general goal is to integrate local information and the
relationship between these pieces of local information in order to obtain the global information
of the dataset; for example, the ptychographic imaging problem \cite{Marchesini_Tu_Wu:2014}, the synchronization problem \cite{Bandeira_Singer_Spielman:2013}, the vector nonlocal
mean/median \cite{lin2018manifold}, the orientability problem \cite{Singer_Wu:2011}, etc. Numerically, the VDM depends on the
spectral study of the graph connection Laplacian  \cite{Chung_Zhao_Kempton:2013, Chung_Kempton:2013, ElKaroui_Wu:2015a, ElKaroui_Wu:2015b, Singer_Wu:2012, Singer_Wu:2016} , which is a
direct generalization of the graph Laplacian discussed in the spectral graph theory \cite{chung:1997}.
In \cite{Lin_Wu:2018}, we showed that one is able to construct an embedding of a manifold with finite eigenvector fields of the connection Laplacian that is  bi-Lipschitz. 

The main contribution in this paper is to show that one can construct an embedding of a manifold with heat kernels of the connection Laplacian that is an almost isometry. Our approach is inspired by Portegies'  rescaling technique. Note that due to the lack of canonical isometry from $L^2(TM)$ to $\ell^2$, we do not have a similar almost isometry result for embeddings with eigenvector fields. 
In Section \ref{sec:main}, we introduce background on the heat kernels and state the main results. 
In Section \ref{sec:fund_sol}, we review the $C^{k,\alpha}$-harmonic radius and introduce the  rescaled heat kernel $\bar{K}_{TM}$ defined in (\ref{rescaled_kernel}) that satisfies the rescaled parabolic system (\ref{eq:parabolic}).   Following  \cite[Section 4 in Chapter 9]{Friedman:1964}, we construct a fundamental solution $\Gamma$ for (\ref{eq:parabolic}) and show that $\Gamma$ is close to the standard Euclidean heat kernel on a local harmonic chart.   We also show that any solution with sufficient decay is close to this fundamental solution $\Gamma$. Thus, the rescaled heat kernel $\bar{K}_{TM}$ is close to the stand Euclidean heat kernel.  In Section \ref{sec:embedding}, we prove our main Theorems. 

\section{Acknowledgments}
I would like to thank Hau-Tieng Wu for the support as well as the helpful discussions and comments. I would also like to thank Duke University for the research environment. I am also supported by PSC-CUNY 63043-00 51.

\section{Background and Main Results}
\label{sec:main}
In this section, we review some background that is required for the main results, as well as the Gaussian bound for the heat kernel that is used in the proofs of the main results.
\subsection{Heat kernel of connection Laplacian}
Let $M$ be an $n$-diimensional smooth closed manifold. 
We denote by $K_{TM}(p,t;q)$ the heat kernel of the connection Laplacian $\Delta_{TM}$ on the manifold $M$.
The heat kernel can be expressed as 
\[
K_{TM}\left(p,t; q\right)=\sum_{i=1}^{\infty} e^{-\lambda_{i}t}X_{i}(p)\otimes X_{i}(q),
\]
where $t>0$, $p,q \in M$, $\lambda_i$ are eigenvalues so that $0 \leq \lambda_1 \leq \lambda_2 \leq \cdots$ and $\{X_i\}_{i=1}^{\infty}$ is an  $L^2(TM)$-orthonormal basis formed by eigenvector fields of $\Delta_{TM}$.
We define the {\em $m$-th order truncated heat kernel} $K_{TM}^{(m)}$ by
\begin{equation}
K_{TM}^{(m)}(p,t;q) := \sum_{i=1}^m e^{-\lambda_{i}t}X_{i}(p)\otimes X_{i}(q). \nonumber
\end{equation}
Let  $\| \cdot \|_{HS}$ denote the Hilbert-Schmidt norm. We have

\begin{align*}
\left\Vert K_{TM}\left(p,t; q\right)\right\Vert _{HS}^{2} 
&= \mathrm{Tr}\left( K_{TM}(p,t; q)^* K_{TM}(p,t;q) \right) \\
&= \sum_{i,j=1}^{\infty}e^{-\left(\lambda_{i}+\lambda_{j}\right)t}\left\langle X_{i}(p),X_{j}(p)\right\rangle \left\langle X_{i}(q),X_{j}(q)\right\rangle; 
\end{align*}
and
\begin{equation*}
\left\Vert K_{TM}^{(m)} \left(p,t; q\right)\right\Vert _{HS}^{2}=\sum_{i,j=1}^{m}e^{-\left(\lambda_{i}+\lambda_{j}\right)t}\left\langle X_{i}(p),X_{j}(p)\right\rangle\left\langle X_{i}(q),X_{j}(q)\right\rangle.
\end{equation*}

\subsection{Dilatation}
Given two metric spaces $(\mathfrak{M}_1, d_1)$ and ($\mathfrak{M}_2,d_2)$, the local dilatation of a map $f: \mathfrak{M}_1 \rightarrow \mathfrak{M}_2$ at a point $p\in \mathfrak{M}_1$ is defined as 
\begin{equation}
\mathrm{dil}_p(f) := \lim_{ r \rightarrow 0} \sup_{x,y \in B_r(p)} \frac{d_2(f(x),f(y))}{d_1(x,y)}\,, \nonumber
\end{equation}
where $B_r(p)$ is a ball in $\mathfrak{M}_1$ of radius $r$ and center $p$.
In the special case that $M$ is a smooth Riemannian manifold mapped into a normed, finite-dimensional vector space via a smooth map $f$, the dilatation is given by 
\begin{equation}
\mathrm{dil}_p(f) = |(df)_p| \nonumber\,,
\end{equation}
where the norm of the right-hand side is interpreted as the operation norm of the map from $T_pM$ to $T_{f(p)}f(M)$. 

\subsection{Main results}

Consider a set of closed smooth manifolds of dimension $n$ 
\be
\mathcal{M}_{n,\kappa, i_0,V} := \{ (M^n,g) :  |Rc|, |\nabla Rc| \leq \kappa ,\, \mathrm{inj(M)} \geq i_0,\,\mathrm{Vol}(M)\leq V\}, \nonumber
\ee
where $\mathrm{inj}(M)$ denotes the injectivity radius of $M$ and $Rc$ and $\nabla Rc$ denote the Ricci curvature and its covariant derivative, respectively.

We prove the following theorems:

\begin{thm}
\label{thm:embedding}
Given $\epsilon >0$, there exists $t_0 = t_0(n, \kappa, i_0, \epsilon)$ so that for all $t< t_0$ there  exists $\delta = \delta(n, \kappa, i_0, \epsilon, t)$ such that for all $(M,g) \in \mathcal{M}_{n,\kappa, i_0,V}$ and every $\delta$-net $\{q_1, \cdots, q_{N_0}\}$ on $M$, the map
\begin{align}
H : M^n &\rightarrow \mathbb{R}^{N_0}  \nonumber \\
 p &\mapsto \frac{(2t)^{(3n+2)/4}}{V_e}\left( |A_1|^{1/2} \| K_{TM} (p, t; q_1) \|^2_{HS}, \cdots, |A_{N_0}|^{1/2} \| K_{TM} (p, t; q_{N_0}) \|^2_{HS}\right) 
\end{align}
is an embedding and the dilatation is controlled by a given $\epsilon$,
\begin{equation}\label{almostiso}
1 - \epsilon < |(dH)_p|^2 < 1+ \epsilon  \quad \mbox{for all $p \in M$}.
\end{equation}

Here $\{ A_i \}_{i=1}^{N_0}$ is a partition of $M$ so that $A_i \subset B_{\delta}(q_i)$ for all $i$ and 
\begin{equation}
V_ e:= \left( \int_{\mathbb{R}^n} (\partial_{x_1} \| \Gamma_E(0,\frac{1}{2}; y)\|^2_{HS})^2 dy \right)^{1/2}\,,
\end{equation}
where $\Gamma_E$ is the standard Euclidean heat kernel on $\mathbb{R}^n\times \mathbb{R}$ whose definition is reviewed in \eqref{Definition standard Euclidean heat kernel}.

Moreover, there exists an $N = N(n,\kappa, i_0, V, \epsilon, t) $ so that for all $m \geq N$, the same statements hold for the map 
$H^{(m)}$ defined by truncated heat kernels
\begin{equation} \label{finite}
H^{(m)}(p) :=   \frac{(2t)^{\frac{(3n+2)}{4}}}{V_e}\left( |A_1|^{\frac{1}{2}} \| K^{(m)}_{TM} (p, t; q_1) \|^2_{HS}, \cdots, |A_{N_0}|^{\frac{1}{2}} \| K^{(m)}_{TM} (p, t; q_{N_0}) \|^2_{HS}\right)
\end{equation}
\end{thm} 
We comment that the $N_0$  is not necessary universal. It is for {\bf every} choice of covering sets that gives a $\delta$-net, where $\delta$ depends on the geometry, $\epsilon$ and $t$. 

\begin{thm}
\label{thm:uniform}
Given $\epsilon >0$, there exists $t_0 = t_0(n, \kappa, i_0, \epsilon)$ so that for all $t< t_0$ there  exists $N_0 = N_0(n, \kappa, i_0, \epsilon, t, V)$ such that for all $(M,g) \in \mathcal{M}_{n,\kappa, i_0,V}$, there exist points  $\{q_1, \cdots, q_{N_0}\}$ on $M$ such that the map
\begin{align}
\mathcal{H} : M^n &\rightarrow \mathbb{R}^{N_0}  \nonumber \\
 p &\mapsto \frac{(2t)^{(3n+2)/4}}{V_e} A \left(\| K_{TM} (p, t; q_1) \|^2_{HS}, \cdots, \| K_{TM} (p, t; q_{N_0}) \|^2_{HS}\right) 
\end{align}
where $A = A(n,\kappa, i_0,  V, \epsilon, t)>0$ is a small constant,  is an embedding and the local dilatation satisfies
\begin{equation}\label{almostiso2}
1 - \epsilon < |dH_p|< 1+ \epsilon \quad \mbox{for all $p \in M$.}
\end{equation}
Moreover, there exists an $N = N(n,\kappa, i_0, V, \epsilon, t) $ so that for all $m \geq N$, the same statements hold with every heat kernel $K_{TM}$ replaced by the truncated heat kernel $K^{(m)}_{TM}$,
\begin{align}
\mathcal{H}^{(m)} : M^n &\rightarrow \mathbb{R}^{N_0}  \nonumber \\
 p &\mapsto \frac{(2t)^{(3n+2)/4}}{V_e} A \left(\| K_{TM}^{(m)} (p, t; q_1) \|^2_{HS}, \cdots, \| K_{TM}^{(m)} (p, t; q_{N_0}) \|^2_{HS}\right) 
\end{align}
\end{thm} 
Note that this $N_0$ is universal as we choose a partition that depends on the geometry, $t$ and $\epsilon$.  

In order to prove the main results, we heavily rely on the Gaussian bound of the heat kernel of the connection Laplacian, which is stated below.
\subsection{Gaussian bound  of the heat kernel}

Let $K_M(p,t;q)$ denote the heat kernel of the Laplace-Beltrami operator $\Delta_M$ on $(M,g)$. By the Faber-Krahn inequality \cite[Section 14.2]{Grigoryan:2009} and the exponential decay \cite[Theorem 15.14]{Grigoryan:2009} of the heat kernel $K_M(p,t;q)$, we have   for $p, q$ in a ball of radius less than the $C^{2,\alpha}$-harmonic radius  $r_h$, whose definition is reviewed in Subsection \ref{harmonic radius},
\be
 K_M(p,t;q) \leq \frac{C(n)(1+d(p,q)^2/t)^{n/2}}{(a(n) \min(t, r_h^2))^{n/2}} \exp\left(-\frac{d(p,q)^2}{4t}\right)
\ee
where $d(p,q)$ denotes the geodesic distance between $p, q \in M$ for constants $a(n)>0$ and $C(n)>0$ depending only on $n$.

In addition, the semigroup domination theorem by Hess-Schrader-Uhlenbrock \cite[Theorem 3.1]{Hess_Schrader_Uhlenbrock:1980} states the following in our setting
\begin{equation}
\label{semi_group_domination}
\textrm{Tr} \left( e^{t\Delta_{TM}} \right) \leq n \textrm{Tr}\left(  e^{t \Delta_M} \right), \nonumber
\end{equation} 
which implies that
\begin{align}
\label{est:heat_kernel}
\| K_{TM}(p,t;q) \|_{HS} \leq n K_M(p,t;q) \leq \frac{C(n)(1+d(p,q)^2/t)^{n/2}}{(a(n) \min(t, r_h^2))^{n/2}} \exp\left(-\frac{d(p,q)^2}{4t}\right).
\end{align}

\section{Foundamental Solutions on Charts} 
\label{sec:fund_sol}
The goal in this section is to show that the {\em rescaled heat kernel}, $\bar{K}_{TM}$ defined in (\ref{rescaled_kernel}), is close to the standard heat Kernel $\Gamma_E$. We start with reviewing the harmonic radius and there is a universal lower bound for the harmonic radius that is independent of the point on the manifold. This allows us to work on a harmonic chart of certain size.
In Lemma \ref{lem:sol}, we construct  a fundamental solution $\Gamma$ of the rescaled parabolic system (\ref{eq:parabolic}) that is close to the Euclidean heat kernel using  the parametrix  method as presented in Friedman's book \cite[Section 4 in Chapter 9]{Friedman:1964}.
In Lemma \ref{lem:closeness}, we  show that the rescaled heat kernel $\bar{K}_{TM}$ and the fundamental solution $\Gamma$ are close by applying the Schauder estimates \cite{Friedman:1958} and the semigroup domination theory. Throughout the paper, the constant $C$ varies line by line. 

\subsection{The $C^{k,\alpha}$-Harmonic Radius}
\label{harmonic radius}

Recall that  (see \cite{ Hebey_Herzlich:1998})
 given $Q>1$ and $\alpha \in (0,1)$, the $C^{k,\alpha}$-harmonic radius at $p$ in $M$ is defined as the largest number $r_h = r_h(Q,k,\alpha)(p)$ such that on the geodesic ball $B_{r_h}(p)$ of center $p$ and radius $r_h$, there is a harmonic coordinate chart
$u: B_{r_h}(p)\subset M \rightarrow U\subset \mathbb{R}^n$ such that $ u(p) = 0$, $g_{ab}(p) = \delta_{ab}$
and the metric tensor is $C^{k,\alpha}$ controlled in these coordinates. Namely, if $g_{ab}$, $a,b=1, \cdots , n$ are the components of $g$ in these coordinates, then 
\begin{align} 
& Q^{-1} \delta_{ab} \leq g_{ab} \leq Q \delta_{ab} \mbox{ as bilinear forms  } \label{metric:1}  \\
& \displaystyle \sum_{1 \leq |\beta| \leq k} r_h^{|\beta|} \sup_p |\partial^{\beta} g (p)|  + \sum_{|\beta| = k} r_h^{2+\alpha} \sup_{p\neq q} \frac{|\partial^{\beta}g_{ij}(p) - \partial^{\beta} g_{ij}(q)|}{d(p,q)^{\alpha}} \leq Q-1 \label{metric:2} \,, 
\end{align}
where $d(p,q)$ denotes the geodesic distance on $(M,g)$. The harmonic radius of $(M,g)$, denoted as $r_h(Q,k,\alpha)(M)$, is defined by  
\begin{align}
r_h(Q,k,\alpha)(M) := \inf_{p\in M} r_h(Q,k,\alpha)(p).
\end{align}

For any manifold $(M,g) \in \mathcal{M}_{n,\kappa, i_0,V}$, it is proved by Hebey-Herzlich \cite[Corollary of Theorem 6]{Hebey_Herzlich:1998} that there exist $C^{2,\alpha}$-harmonic coordinate charts whose size is independent of the choice of points.
Here, we restate the  Corollary for our case. 
\begin{cor}
\label{cor:har_rad}
Let $\alpha \in (0,1)$ and $Q>1$. Let $(M,g)$ be a smooth closed $n$-dimensional Riemannian manifold. Suppose for some $\kappa >0 $, $i_0 >0$ 
\be
|Rc|, |\nabla Rc| \leq \kappa \quad\mbox{and}\quad \mathrm{inj}(M) \geq i_0\,. \nonumber
\ee
Then, there exists a constant $C = C(n, Q, \alpha, \kappa, i_0)$ such that the $C^{2,\alpha}$-harmonic  radius 
\begin{equation}
r_h = r_h(Q, 2, \alpha)(M) \geq C.\nonumber
\end{equation}
\end{cor}

%

That is, for any $M \in \mathcal{M}_{n,\kappa, i_0,V}$, there is a universal $C^{2,\alpha}$-harmonic radius $r_h$ that is independent of the choice of point. 
Let $p \in M$ and let  $u: B_{r_h}(p) \rightarrow \RR^n$ be a harmonic coordinate chart with $u(p) = 0$ satisfying (\ref{metric:1}) and (\ref{metric:2}). 
%
For a vector field $X = X^a\partial_a$, by a direct calculation, we have 
\begin{eqnarray}
\label{eq:elliptic}
\Delta_{TM} \left(X^c \partial_c \right)
=- g^{ab}\left(\partial_a \partial_b X^c +2 \Gamma_{bd}^c \partial_a X^d + \Gamma_{ae}^c \Gamma_{bd}^e X^d + \partial_a \Gamma_{bd}^c  X^d \right) \partial_c.
\end{eqnarray}
Note that the coefficients of $\Delta_{TM}  X$ are controlled in $C^{2,\alpha}$-harmonic coordinates by (\ref{metric:2}).

\begin{rmk}
In \cite{Lin_Wu:2018} where we have similar results but without discussing the Lipschitz control,  we write the term $ \partial_a \Gamma_{bd}^c  X^d$ in terms of the Ricci curvature and hence having $|Rc|<\kappa$ is enough to have $C^{1,\alpha}$-harmonic radius and bounded coeffients. However, in this paper we will start with a rougher guess in the parametrix process to build a solution that is close to the stand Euclidean heat kernel and need subtler controls of the coefficients. 
Therefore, the class of manifold $\mathcal{M}_{n,\kappa, i_0,V}$  requires the extra condition $|\nabla Rc| \leq \kappa$ to guarantee a universal $C^{2,\alpha}$-harmonic radius bound. 
\end{rmk}

\subsection{ Rescaled Heat Kernel}
Let  
$u: B_{r_h}(p)\subset M \rightarrow U\subset \mathbb{R}^n$
be  a harmonic coordinate chart with $u(p)=0$. 
Define a {\em rescaled heat kernel} on $\RR^n \times \RR_+$ associated with the heat
kernel $K_{TM}$ of the connection Laplacian as:
\begin{equation}
\label{rescaled_kernel}
\bar{K}_{TM}(x,s;y) := r^n K_{TM}( u^{-1}(xr), sr^2; u^{-1}(yr)),
\end{equation}
where $s > 0$, $r>0$ is a scale factor to be chosen later, and $x, y$  in the proper domain, $U/r:=\{x\in \mathbb{R}^d|\, rx \in U\}\subset \mathbb{R}^n$.
Since $u$ is a differomorphism, denote $(u^{-1})^*g$ to be the induced metric on $U$, and denote the rescaled metric as $\bar{g}(x) = (u^{-1})^*g(rx)$. Denote the Christophe symbol of $(u^{-1})^*g$ as $\Gamma_{ab}^c$, we have the Christophe symbol for the rescaled metric satisfying $\bar{\Gamma}_{ab}^c(x) = \Gamma_{ab}^c(rx)$. On $(U, (u^{-1})^*g)$, denote $L=\Delta_{(u^{-1})^*g}-\partial_s$ to be the parabolic system, and $\bar{L}$ to be the rescaled version. By a direct calculation, it satisfies the following expansion:
 \begin{equation}
\label{eq:parabolic}
\bar L X^c  :=  \bar{ g}^{ab}\left(\partial_a \partial_b X^c +2 r\bar{ \Gamma}_{bd}^c \partial_a X^d + r^2\bar{ \Gamma}_{ae}^c \bar{\Gamma}_{bd}^e X^d +  r^2 \partial_a \bar{\Gamma}_{bd}^c  X^d \right)- \partial_s X^c =0, 
\end{equation}
where  $c = 1 \cdots, n$ and $X^c$ is a function defined on $U/r$.

We consider the standard Euclidean heat kernel on  $\RR^n \times \mathbb{R}$ given by 
\begin{equation}
\Gamma_E(x, t ; y) = \frac{1}{(4\pi t)^{n/2}} e^{-\frac{|x-y|^2}{4t}} I_n, \label{Definition standard Euclidean heat kernel}
\end{equation}
where $t>0$, $x,y\in \mathbb{R}^n$, and $I_n$ is the $n \times n$ identity matrix. Note that $\Gamma_E$ is the heat kernel of the connection Laplacian associated with the trivial tangent bundle of the canonical Euclidean space $\mathbb{R}^n$.
Below, we  show that $\bar{K}_{TM}$ is close to the standard Euclidean heat kernel $\Gamma_E$.

\subsection{Foundamental solutions on charts} 

Denote $P_{R,T}(x)$ to be the parabolic cylinder on $\mathbb{R}^n \times \mathbb{R}$; that is, 
\begin{equation}
P_{R,T}(x) := \left\{ (y,s) \in \RR^n \times \RR^+ : 0< s <T, |y-x|  < R  \right\}\,,  \nonumber
\end{equation}
where $x \in \RR^n$,  $R,T >0$.

Define $Z(x,s;y)$ on $\mathbb{R}^n$ as 
\begin{equation} \label{eq:Z}
Z(x,s;y) : = \frac{\sqrt{|\bar{g}(y)|}}{(4 \pi s)^{n/2}} \exp\left( - \frac{\bar{g}_{ab}(y) (x^a - y^a) (x^b - y^b)}{4s} \right) I_n,
\end{equation} 
where $|\bar{g}(y)| := \det(\bar{g}_{ab}(y))$. 
 Note that $Z(x,s;y)$
is a solution of the principal part of the rescaled system  $ \partial_s X^c  = \bar{g}^{ab}(y)\partial_a \partial_b X^c$.

\begin{lem} 
\label{lem:sol}
Let $\epsilon>0$, $0 < \alpha <1$, $1 < Q < \sqrt{2}$ and $R>0$ be given. Assume that  $r$ is  small so that $rR < r_h/ \sqrt{2}$, where $r_h$ is the $C^{2,\alpha}$-harmonic radius. Consider the rescaled parabolic system (\ref{eq:parabolic}) on a domain $\Omega \times I \subset \RR^n \times \RR_+$ that contains the rescaled domain $P_{R,4}(0)$. 
There exists a fundamental solution $\Gamma$ for (\ref{eq:parabolic}) on the domain $\Omega \times I$ such that
\be \label{eq:Gauss_decay}
|\Gamma(x,s;y)| \leq \frac{C(n)}{s^{n/2}}e^{-\frac {|x-y|^2}{8s}} \quad  \mbox{ for $0<s<2$}.
\ee
In particular,
 for all $y\in B_R(0) \subset \Omega$ and $(x,s) \in P_{R,4}(0)$,
\begin{align}
\label{eq:sol_compare_to_Z}
|\Gamma(x,s;y) - Z(x,s;y)| & \leq \frac{(Q-1)C(n, \alpha)}{s^{(n-\alpha)/2}}e^{-\frac{|x-y|^2}{8s}}; \\
|\nabla \Gamma(x,s;y) - \nabla Z(x,s;y)| & \leq \frac{(Q-1)C(n, \alpha)}{s^{(n+1-\alpha)/2}}e^{-\frac{|x-y|^2}{8s}} .
\end{align}
\label{eq:grad_sol_compare_to_Z}
Moreover, for $(x,s) \notin P_{\frac 12, \frac 14}(y)$,
\be
\label{eq:sol_close_to_Euc}
| \Gamma(x,s;y) - \Gamma_E(x,s;y)| \leq (Q-1)C(n,\alpha);
\ee
\be
\label{eq:grad_sol_close_to_Euc}
|\nabla \Gamma(x,s;y) - \nabla \Gamma_E(x,s;y)| \leq (Q-1)C(n,\alpha).
\ee
\end{lem}

\begin{proof}
Following the parametrix method in Friedman  \cite[Section 4 in Chapter 9 and Section 4 in Chapter 1]{Friedman:1964},
%
we construct $\Gamma$ in the form
\begin{equation} \label{eq:fund}
\Gamma ( x,s ;y) = Z(x,s;y)  + \int_0^s \int_{\Omega} Z(x, s-\tau ; \xi) \Phi(\xi, \tau ;y) d\xi d\tau \nonumber
\end{equation}
for some $\Phi$ which is an $n \times n$ matrix-valued function defined on $\Omega \times I$. 
If $\Phi$ is H\"{o}lder continuous, then 
$\Gamma$ satisfies the system (\ref{eq:parabolic}) as a function of $(x,s)$ if and only if 
\begin{equation}\label{fund_sol:1}
\Phi(x, s; y) := L Z(x, s; y) + \int_0^s \int_{\Omega} LZ(x, s-\tau; \xi) \Phi(\xi ,\tau; y) d\xi d\tau.
\end{equation}
(See  \cite[Lemma 5 in Chapter 9, Section 4, p. 250]{Friedman:1964} for details.)
We show that there is a $\Phi$ of the form 
\begin{equation}
\Phi(x, s; y) = \sum_{i=1}^\infty (LZ)_i(x, s; y), \nonumber
\end{equation}
where $(LZ)_1 = LZ$ and 
\begin{equation}
(LZ)_i = \int_0^s \int_{\Omega} LZ(x, s-\tau; \xi) (LZ)_{i-1}(\xi, \tau;  y) d\xi d\tau , \nonumber
\end{equation}
so that $\Phi$ is a formal solution of (\ref{fund_sol:1}). To find such $\Phi$, we first show that $(LZ)_i$ is integrable for all $i \in \NN$ and then show the convergence of the series.  
Since $Z$ is diagonal and identical along the diagonal. It suffices to look at the scalar function
\begin{equation}
z(x,s - \tau ;y) := \frac{\sqrt{|\bar{g}(y)|}}{(4 \pi (s - \tau))^{n/2}} \exp\left( - \frac{\bar{g}_{ab}(y) (x^a - y^a) (x^b - y^b)}{4(s - \tau)} \right) . \nonumber
\end{equation}
For simplicity, denote $A :=  \bar{g}_{ab}(y) (x^a - y^a) (x^b - y^b)$.
By rewriting
\begin{align}
\exp\left( - \frac{A}{4(s - \tau)}  \right) = \exp\left( - \epsilon \frac{A}{4(s - \tau)}  \right)\exp\left( - (1-\epsilon)\frac{A}{4(s - \tau)}  \right) , \nonumber
\end{align}
and using the inequality $\sigma^{n/2 - \mu} e^{-\epsilon \sigma} \leq \mathrm{constant}$ for fixed constants $\mu, \epsilon$, and $0\leq \sigma <\infty$, we obtain a bound for $z(x,s - \tau; y)$:
\begin{align}
|z(x, s - \tau; y)|
&=  \frac{\sqrt{|\bar{g}(y)|}}{ \pi^{n/2}(4(s -\tau))^{\mu}}  A^{\mu - n/2}  \left(\frac{A}{4 (s-\tau)}  \right)^{n/2 -\mu} \exp\left( - \frac{A}{4(s - \tau)}  \right)  \nonumber\\
& = \frac{\sqrt{|\bar{g}(y)|}}{ \pi^{n/2}(4(s -\tau))^{\mu}}  A^{\mu - n/2}  \left(\frac{A}{4 (s-\tau)}  \right)^{n/2 -\mu} \nonumber \\
& \qquad \qquad \times
  \exp\left( - \epsilon \frac{A}{4(s - \tau)}  \right)\exp\left( - (1-\epsilon)\frac{A}{4(s - \tau)}  \right)  \nonumber\\
& \leq \frac{C}{(s-\tau)^{\mu} | x-y|^{n-2\mu}} \exp\left( - \frac{\lambda_0 |x-y|^2}{4(s - \tau)} \right) , \nonumber
\end{align}
for any $\lambda_0 < Q^{-1}, 0 \leq \mu \leq n/2$, and some constant $C=C(n,Q)$. Similarly, it can be proved that 
\begin{align}
\sum_{c=1}^{n} \left| \frac{\partial z(x, s - \tau; y)}{\partial x^c} \right| \leq \frac{C}{(s-\tau)^{\mu} | x-y|^{n+1-2\mu}} \exp\left( - \frac{\lambda_0 |x-y|^2}{4(s - \tau)} \right)  \nonumber
\end{align}
for any $\lambda_0 < Q^{-1}$ and $0 \leq \mu \leq (n+1)/2$. 
Since that $Z$ is a solution to the principal part of (\ref{eq:parabolic}) and that  the coefficients of the  rescaled parabolic system (\ref{eq:parabolic}) is controlled under the  $C^{2,\alpha}$ harmonic coordinates, we further have 
\begin{align}
\left| \frac{\partial z(x, s - \tau; y)}{\partial s} \right| + \sum_{c, d=1}^{n} \left| \frac{\partial z(x, s - \tau; y)}{\partial x^c \partial x^d} \right|
\leq  \frac{(Q-1)C(n, \alpha)}{(s-\tau)^{\mu} | x-y|^{n+2-2\mu}} \exp\left( - \frac{\lambda_0 |x-y|^2}{4(s - \tau)} \right) \nonumber
\end{align}
for any $\lambda_0 < Q^{-1}$ and $0 \leq \mu \leq (n+2)/2$. 
Therefore, under the rescaled $C^{2, \alpha}$ coordinate chart, we have 
\begin{align}
\label{est:LZ}
|LZ(x,s-\tau; y)| \leq \frac{(Q-1)C(n, \alpha)}{(s-\tau)^{\mu} | x-y|^{n+2-2\mu - \alpha}} \exp\left( - \frac{\lambda_0 |x-y|^2}{4(s - \tau)} \right)
\end{align}
for any $\lambda_0 < Q^{-1}$ and $0 \leq \mu \leq (n+2 - \alpha)/2$. Hence, the singularity $s = \tau, x=y$ is integrable when $1-\alpha/2 \leq \mu \leq (n+2 - \alpha)/2$. 
Recall the integral formula in \cite[Lemma 3]{Friedman:1964}, that is, for fixed constants $\alpha$ and $\beta$, 
\begin{align} \label{formula:integral}
&\int_0^s \int_{\Omega} (s-\tau)^{-\alpha} e^{-\frac{\lambda_0|x - \xi|^2}{4(s-\tau)}} \tau^{-\beta} e^{-\frac{\lambda_0| \xi -y|^2}{4\tau}} d\xi d\tau \nonumber \\
\leq &\, \left(\frac{4 \pi}{\lambda_0} \right)^{n/2} \frac{\Gamma(n/2-\alpha+1) \Gamma(n/2-\beta+1)}{\Gamma(n-\alpha-\beta+2)} s^{n/2+1 -\alpha - \beta} e^{-\frac{\lambda_0|x - y|^2}{4s}},
\end{align}
where $\Gamma(\cdot)$ denotes the gamma function. 
Using (\ref{est:LZ}) with $\mu = (n + 2 - \alpha)/2$ and applying (\ref{formula:integral}), we find
\begin{align}
|(LZ)_2 (x,s-\tau; y)| \leq \frac{(Q-1)C(n, \alpha)}{(s-\tau)^{(n+2 - 2\alpha)/2}} \exp\left( - \frac{\lambda_0 |x-y|^2}{4(s-\tau) } \right) \nonumber
\end{align}
Proceeding by induction, we further obtain
\begin{align}
|(LZ)_i (x,s-\tau; y)| \leq \frac{(Q-1)C(n, \alpha)}{\Gamma(i\alpha)(s-\tau)^{(n+2 - i\alpha)/2}} \exp\left( - \frac{\lambda_0 |x-y|^2}{4(s - \tau)} \right) \nonumber\,,
\end{align}
where $\lambda_0 < Q^{-1}$. Note that the coefficient $C(n,\alpha)$ is independent of $i$. 
It follows that the series expansion of $\Phi$ is convergent and satisfies (\ref{fund_sol:1}).  Furthermore,
\begin{equation}
|\Phi(x,s;y)| \leq \frac{(Q-1) C(n,\alpha) }{s^{(n+2-\alpha)/2}}e^{-\frac{|x-y|^2}{8s}}  \nonumber
\end{equation}
by setting $\lambda_0= 1/\sqrt{2}$.

The H\"{o}lder continuity of $\Phi(x,s;y)$ in $x$  is the same as the proof in \cite[Section 4 in Chapter 9 and Section 4 in Chapter 1]{Friedman:1964} (see (1.4.17) and  (9.4.17)). Therefore, we omit the details and have $\Gamma$ the fundament solution of (\ref{eq:parabolic}).

By (\ref{formula:integral}) and a trivial bound of $|Z(x, s-\tau;  \xi)|$ under the $C^{2,\alpha}$-harmonic coordinate, it follows that
\begin{equation}
\int_0^s \int_{\Omega} |Z(x, s-\tau;  \xi) \Phi(\xi, \tau; y) | d\xi d\tau \leq \frac{(Q-1)C(n, \alpha) }{s^{(n-\alpha)/2}} e^{-\frac{|x-y|^2}{8s}},  \nonumber
\end{equation}
and hence (\ref{eq:sol_compare_to_Z}).
In particular, we have (\ref{eq:Gauss_decay}),  for $s \leq 4$,
\begin{equation}
|\Gamma(x, s; y)| \leq \frac{D(n,  \alpha)}{s^{n/2}} e^{-\frac{|x-y|^2}{8s}}  \nonumber
\end{equation}
for some constant $D(n,\alpha)$. 
Note that 
\begin{equation}
\nabla \Gamma(x, s; y) = \nabla Z(x, s; y) + \int_0^s \int_{\Omega} \nabla Z(x, s-\tau; \xi) \Phi( \xi, \tau; y) d\xi d\tau.  \nonumber
\end{equation}
Applying (\ref{formula:integral}) once again, it follows that
\begin{equation}
\int_0^s \int_{\Omega} |\nabla Z(x, s-\tau; \xi)||\Phi(\xi, \tau; y)| d\xi d\tau \leq   \nonumber
\frac{(Q-1)C(n, \alpha) }{s^{(n+1-\alpha)/2}} e^{-\frac{|x-y|^2}{8s}}.
\end{equation}
and hence (\ref{eq:grad_sol_compare_to_Z}) holds.

Ineqaulities (\ref{eq:sol_close_to_Euc}) and (\ref{eq:grad_sol_close_to_Euc}) follow from straightforward compuations that 
\be
\frac {\partial}{\partial x^a} Z(x,s;y) = \frac 1{2s} \bar{g}_{ab}(y)(x^b-y^b) Z(x,s;y)  \nonumber
\ee
and
\be
\frac {\partial}{\partial x^a} \Gamma_E(x,s;y) =  \frac 1{2s}\delta_{ab} (x^b - y^b) \Gamma_E(x,s;y).  \nonumber
\ee
It follows that for $(x,s) \notin P_{\frac 12, \frac 14}(y)$, 
\be
| \Gamma_E(x,s;y) - Z(x,s;y)| \leq (Q-1)C(n,\alpha)  \nonumber
\ee
and 
\be
|\nabla \Gamma_E(x,s;y) - \nabla Z(x,s;y)| \leq (Q-1)C(n,\alpha)  \nonumber
\ee
and hence  (\ref{eq:sol_close_to_Euc}) and (\ref{eq:grad_sol_close_to_Euc}). 
\end{proof}

With the interior Schauder estimates, the semigroup domination, and the maximum principle on the heat equation $\partial_t u = \Delta_M u$, we have Lemma \ref{lem:closeness} which shows that every solution to (\ref{eq:parabolic}) that decays exponentially is close to  $\Gamma$ that we construct in Lemma \ref{lem:sol}, and hence is close to the standard Euclidean heat kernel $\Gamma_E$.

\begin{lem} \label{lem:closeness}
Let $\epsilon >0$, $C(n)>0$ and $0<\alpha <1$, $1 < Q < \sqrt{2}$, and $r$  be given. Suppose $\Gamma_{TM}^1$ and $\Gamma_{TM}^2$ are fundamental solutions of (\ref{eq:parabolic}) on $\RR^n \times \RR_+$ satisfying the decay rate
\begin{equation}\label{Exponential decay of GammaTM}
|\Gamma_{TM}^i (x,s; y)|\leq \frac{C(n)}{s^{n/2}}e^{-\frac{|x-y|^2}{8t}}, \quad \mbox{for $i = 1, 2$ and $0<s\leq4$}.
\end{equation} 
There exists $\bar{R} =\bar{R}(n, \alpha, Q, C(n), \epsilon)$ so that for $R > \bar{R}$,
\begin{equation} \label{diff_bd:1}
| \Gamma_{TM}^1(\cdot, \cdot ; y) - \Gamma_{TM}^2(\cdot, \cdot ; y) |_{C^{2,\alpha}(P_{2,4}(0))} \leq \epsilon
\end{equation}
for all $y\in B_R(0)$.
\end{lem}

\begin{proof}
By the Schauder type interior estimates \cite[Theorem 3]{Friedman:1958}, there is a constant $C = C(n, \alpha, Q)$ so that for $1\leq a, b \leq n$,
\begin{equation}
| (\Gamma_{TM}^i)_{ab}(\cdot, \cdot ;  y)   |_{C^{2,\alpha}(P_{2,4}(0))} \leq 
C\left| \sum_{b=1}^n (\Gamma_{TM}^i)_{ab}(\cdot, \cdot ; y) \right|_{C^0(P_{3,4}(0))} .
\end{equation}
By the semigroup domination theory \cite[Theorem 3.1]{Hess_Schrader_Uhlenbrock:1980} 
\begin{equation}
\textrm{Tr} \left( e^{t\Delta_{TM}} \right) \leq n \textrm{Tr}\left(  e^{t \Delta_M} \right),
\end{equation} 
and the fact that the associated solutions $\Gamma_M^i(x, s; y)$, $i=1,2$, of $\partial_s u = \Delta_M u$ decay exponential $\frac{C(n)}{s^{n/2}}\exp\big( -\frac{|x-y|^2}{8s}\big)$, we have the bound \eqref{Exponential decay of GammaTM}.

By the maximum principle, there exists $\bar{R} = (n, \alpha, Q, C, \epsilon)$ so that for $y \in B_R(0)$, $R > \bar{R}$,
\begin{align}
& | \Gamma^1_M(\cdot, \cdot ; y) - \Gamma^2_M(\cdot, \cdot ; y) |_{C^{0}(P_{3,4}(0))} \nonumber \\
&\leq | \Gamma^1_M(\cdot, \cdot ; y) - \Gamma^2_M(\cdot, \cdot ; y) |_{C^0(\partial B_R(0) \times [0,4])} 
< \epsilon.
\end{align}
Combine this with \eqref{Exponential decay of GammaTM}, we have (\ref{diff_bd:1}).
\end{proof}
By Kato's inequality and Gaussian bounds on $K_M$, we have  (\ref{est:heat_kernel}), that is, $\bar{K}_{TM}$ has a  Gaussian upper bound.
Thus, by Lemmas \ref{lem:sol} and \ref{lem:closeness}, $\bar{K}_{TM}$ is close to $\Gamma_E$. 

\section{Embedding with Heat Kernels}
\label{sec:embedding}
In this section, we prove Theorems \ref{thm:embedding} and \ref{thm:uniform} that manifolds can be embedded with heat kernels as well as  how the local dilatation can be controlled by taking the fast decay of heat kernels and  the rescaling technique into account.  

\begin{proof}
 Let $\epsilon$ be given. 
To determine a scale $r>0$, we need the following three steps.

Step 1.
 By the Schauder interior estimates  \cite[Theorem 3]{Friedman:1958} for parabolic systems, 
$$
|K_{TM}( u^{-1}(\cdot), \cdot ;q)|_{C^{2+\alpha} (P_{\frac {r_h}{8},T}(0))}\leq C(n)  |K_{TM}(u^{-1}(\cdot), \cdot; q)|_{C^0(P_{\frac {r_h}{4},T}(0))}
$$
on a local domain in $\mathbb{R}^n \times\mathbb{R}_+$. 
By the Gaussian bound of the heat kernel  (\ref{est:heat_kernel}),  we can select  $0< r_0 = r_0(n, \kappa, i_0, \epsilon) < \frac {r_h}{2}$ so that for $0<t < 2r_0^2$,
\be
\label{grad_control}
(2t)^{\frac {3n+2}{2}} \int_{M \setminus B_{\frac{r_h}{2}}(p)}\|\nabla_{q} | K_{TM}(p,t;q)\|_{HS}^2| dq < \epsilon.
\ee

Step 2. Let $R_1 = R_1(n, \epsilon)$ be a radius so that for every $ \Gamma(x,s;y)$ satisfying the  gradient decay
\begin{align}\label{eq:decay}
| \nabla_{y} \| \Gamma(x,s;y) \|^2_{HS} | & \leq \frac{D(n)}{s^{n+1/2}}\exp\left( -\frac{|x-y|^2}{8s} \right) \mbox{ on $(\mathbb{R}^n \setminus B_{\frac{R_1}{\sqrt{2}}}(0)) \times \mathbb{R}$},
\end{align}
the integral
\be
\label{grad_decay}
\int_{\mathbb{R}^n \setminus B_{\frac{R_1}{\sqrt{2}}}(0)} | \nabla_{y} \| \Gamma (0,s;y)\|_{HS}^2| dy < \epsilon \quad \mbox{for  $\frac 12 \leq s \leq 2$. }
\ee

Step 3. Set $\alpha = \frac 12$. Choose $1 < Q < \sqrt{2}$ so that $Q-1 < \epsilon$.  Select $R_2 = R_2(n, \alpha = \frac 12, Q, \epsilon)$ as in Lemma \ref{lem:closeness}
such that the rescaled heat kernel $\bar{K}_{TM}(0,\frac 12; y)$ and the solution $\Gamma(0,\frac 12;y)$ obtained in Lemma \ref{lem:sol}  are close for all $y \in B_{R_2}(0)$. 

We now set 
\be 
r_1 := \min \left(r_0,  \frac{r_h}{2R_1}, \frac{r_h}{2R_2} \right)
\ee
and choose a scale $r< r_1$ and $t = \frac {r^2}{2}$. 

Let $\omega$ denote the modulus of continuity of $\nabla \| \Gamma_E(0,s;\cdot)\|^2_{HS}$ .
Let $\delta$ be smaller than $ r \omega^{-1}(\epsilon)/2$  and $\{ q_i\}_{i=1}^{N_0}$ be a $\delta$-net. Let $\{A_i\}_{i=1}^{N_0}$ be a partition of the manifold $M$ so that $A_i \subset B_{\delta}(q_i)$.

 Let $p \in M$. With  the chosen $\alpha =\frac{1}{2}$ and $Q$, there exists a $C^{2,\alpha}$-harmonic coordinate chart $u: B_{r_h}(p) \rightarrow \RR^n$  with $u(p)=0$. 
Let $v = \sum_{a=1}^n v^a \partial_a \in T_pM$, $|v| =1$  so that $\hat{v} = (v^1, \cdots, v^n)$ is the coordinate of $v$ under the harmonic coordinate chart.
We next show that for any $t<2r_0^2$, where $r_0 = r_0(n, \kappa, i_0, \epsilon)$, there exists a constant $\delta=\delta(n,\kappa, i_0, \epsilon, t)$ so that for every $\delta$-net, $ |dH_p(v)|^2 $ is close to $1$. 

Let $I_{\rho}(p)$ denote the subset of $\{ 1, \cdots, N_0 \}$ such that $A_j \cap B_{\rho}(p) \neq \emptyset$ and  $y_i := r^{-1}u(q_i)$. 
 Then, since  $\bar{K}_{TM}(x,s;y) := r^n K_{TM}( u^{-1}(xr), sr^2; u^{-1}(yr))$,
\begin{align*}
 |dH_p(v)|^2 
&= \frac{(2t)^{\frac{3n+2}{2}}}{V^2_e}  \sum_{i=1}^{N_0} |A_i|  (\nabla \| K_{TM} (p, t; q_i) \|^2_{HS} \cdot v)^2 \\
& =  \frac{1}{V^2_e} \sum_{i \in I_1}\frac{ |A_i|}{r^n} (\nabla \| \bar{K}_{TM} (0, \frac{1}{2};y_i) \|^2_{HS} \cdot \hat{v})^2 \\
& \quad +  \frac{1}{V^2_e}  \sum_{i \in I_2}\frac{ |A_i|}{r^n}  (\nabla \|\bar{ K}_{TM} (0, \frac{1}{2}; y_i) \|^2_{HS} \cdot \hat{v})^2 \\
& \quad  +     \frac{(2t)^{\frac{3n+2}{2}}}{V^2_e} \sum_{i\in I_3} |A_i|   (\nabla \| K_{TM} (p, t; q_i) \|^2_{HS} \cdot v)^2 \\
& =\mbox{ (I) + (II) + (III)}
\end{align*}
where $I_1 = I_{rR_1}(p)$, $I_2 = I_{r_h/2}(p)\setminus I_1$, and $I_3 = \{1, \cdots, N_0\} \setminus I_1 \cup I_2$.  Here, we write the first two terms in terms of  the rescaled coordinates.

First, note that $\cup_{i \in I_3} A_i$ lies inside $ M \setminus B_{\frac {r_h}{2}}(p)$.  Because of  equation (\ref{grad_control}),  (III) is controlled by $\epsilon$, that is,
\be \label{I3}
   (2t)^{\frac{3n+2}{2}} \sum_{i\in I_3} |A_i|   (\nabla \| K_{TM} (p, t; q_i) \|^2_{HS} \cdot v)^2 < C(n) \epsilon.
\ee

Second, since $\mathrm{diam}( A_i) < r\omega(\epsilon)$, 
\begin{align*}
\left|  \sum_{i \in I_2}\frac{ |A_i|}{r^n}  (\nabla \|\bar{ K}_{TM} (0, \frac{1}{2}; y_i) \|^2_{HS} \cdot \hat{v})^2 
   - \sum_{i\in I_2} \int_{\frac{u(A_i)}{r}}  (\nabla \|\bar{ K}_{TM} (0, \frac{1}{2}; y) \|^2_{HS} \cdot \hat{v})^2 d\mu    \right| < C(n)\epsilon
\end{align*}
where  $d\mu$ denotes the push-forward of the standard volume measure  under $r^{-1}u$.

Note that  $\cup_{i \in I_2} A_i$ lies outside of $B_{rR_1}(p)$ and inside $B_{r_h}(p)$. Thus, $\cup_{i \in I_2}r^{-1} u(A_i)$ is a subset of $\mathbb{R}^n \setminus  B_{\frac{R_1}{\sqrt{2}}}(0)$.
By Lemmas \ref{lem:sol} and \ref{lem:closeness}, $\bar{K}_{TM}(x,\frac{1}{2}; y)$  satisfies the gradient decay (\ref{eq:decay}).  By the choice of $R_1$ in (\ref{grad_decay}), we have that
$$
\sum_{i\in I_2} \int_{\frac{u(A_i)}{r}}  (\nabla \|\bar{ K}_{TM} (0, \frac{1}{2}; y) \|^2_{HS} \cdot \hat{v})^2 d\mu   < \epsilon
$$
and that (II) is controlled by $\epsilon$ as well, 
\be \label{I2}
 \sum_{i \in I_2}\frac{ |A_i|}{r^n}  (\nabla \|\bar{ K}_{TM} (0, \frac{1}{2}; y_i) \|^2_{HS} \cdot \hat{v})^2 < C(n) \epsilon.
\ee

Last, we show that (I) is close to $1$. 
Let 
\begin{align}
R :=  \max(R_1, R_2).
\end{align}
 Note that $B_R(0) \subset u(B_{r_h}(p))/r$.
By Lemma \ref{lem:closeness} and our choice of $Q$ and $R_2$, for all $y \in B_R(0)$, 
$\bar{K}_{TM}(0,\frac 12; y)$ is close to the solution $\Gamma(0,\frac 12; y)$ constructed in Lemma \ref{lem:sol}, which is close to the Euclidean kernel $\Gamma_E(0, \frac 12; y)$. Specifically,
\begin{equation*}
\left| \nabla \| \bar{K}_{TM}(0,\frac{1}{2}, y) \|^2_{HS} - \nabla \| \Gamma_E(0, \frac{1}{2}, y) \|^2_{HS} \right|  < C(n) \epsilon. 
\end{equation*}
This implies
\begin{align} \label{gamma:1}
 &\left| \sum_{i \in I_1} \left( ( \hat{v} \cdot \nabla \| \bar{K}_{TM}(0, \frac{1}{2}; y_i ) \|^2_{HS} )^2  -  ( \hat{v} \cdot \nabla \| \Gamma_E(0, \frac{1}{2}; y_i) \|^2_{HS} )^2 \right) \frac{|A_i|}{r^{n}} \right|  \nonumber \\
&<  C(n) \epsilon \sum_{ i \in I_1} \frac{|A_i|}{r^n} \nonumber \\
& <  C(n) \epsilon |B_{R_1}(0)|.
\end{align}
Since $\mathrm{diam}(A_i) < r \omega(\epsilon)$,  we have 
\begin{align} \label{gamma:2}
 \left|  \sum_{i \in I_1}  ( \nabla_{\hat{v}} \| \Gamma_E( 0, \frac{1}{2} ; y) \|^2_{HS} )^2 \frac{|A_i|}{r^{n}}  - \int_{\frac{u(A_i)}{r}} ( \nabla_{\hat{v}} \| \Gamma_E(0, \frac{1}{2};  y) \|^2_{HS} )^2 d\mu  \right| 
 < C(n) \epsilon, 
\end{align}
where $d\mu$ denotes the push-forward of the standard volume form under $r^{-1}u$.  
Due to that $Q^{-1} < g < Q$ and symmetry of $\Gamma_E(x, t ; y) = \frac{1}{(4\pi t)^{n/2}} e^{-\frac{|x-y|^2}{4t}} I_n$, we have 
\begin{align}  \label{gamma:3}
& \left|\sum_{i \in I_1}  \left( \int_{\frac{u(A_i)}{r}}  (\nabla_{\hat{v}} \| \Gamma_E(0, \frac{1}{2}; y) \|^2_{HS} )^2 d\mu - \int_{\frac{u(A_i)}{r}}  ( \nabla_{\hat{v}} \| \Gamma_E(0, \frac{1}{2}; y) \|^2_{HS} )^2 dy \right) \right|  \nonumber \\
& = \left|\sum_{i \in I_1}  \left( \int_{\frac{u(A_i)}{r}}  (\nabla_{\hat{v}} \| \Gamma_E(0, \frac{1}{2}; y) \|^2_{HS} )^2 d\mu - \int_{\frac{u(A_i)}{r}}  ( \partial_{x_1} \| \Gamma_E(0, \frac{1}{2}; y) \|^2_{HS} )^2 dy \right) \right|   \\
& <  C(n) (Q-1) \epsilon < C(n) \epsilon. \nonumber
\end{align}

Furthermore, by the choice of $R_1$ and (\ref{grad_decay}), 
\be \label{gamma:4}
\left|\sum_{i \in I_1}  \int_{\frac{u(A_i)}{r}}  ( \partial_{x_1} \| \Gamma_E(0, \frac{1}{2}; y) \|^2_{HS} )^2 dy - V_e^2 \right| < \epsilon
\ee

Combining inequalities (\ref{gamma:1}), (\ref{gamma:2}), (\ref{gamma:3}), and (\ref{gamma:4}), we have that (I) is close to $1$. Together with (\ref{I3}) and (\ref{I2}), we conclude that 
\begin{equation*}
|dH_p(v)^2 -1| \leq C(n) \epsilon
\end{equation*}
that is, (\ref{almostiso}) in Theorem \ref{thm:embedding}.

As for the second part, (\ref{finite}) of Theorem \ref{thm:embedding}, by Lemma 10 in \cite{Lin_Wu:2018}, it follows that there exists an $N = N(n, \kappa, i_0, V, \epsilon, t_0)$ so that for any $m \geq N$ and $0<t\leq t_0$, 
\begin{align}
\| \|K_{TM}^{(m)} ( \cdot, t ; q)\|^2_{HS} - \| K_{TM}( \cdot, t ; q)\|^2_{HS} \|_{\infty} & < \epsilon \\
\| \nabla \| K_{TM}^{(m)} ( \cdot, t ; q) \|^2_{HS} - \nabla \| K_{TM}( \cdot, t ; q) \|^2_{HS} \|_{\infty} & < \epsilon.
\end{align}
Therefore,
\begin{equation*}
1- \epsilon < \left|  dH_p^{(m)} (v) |^2 \right| < 1 + \epsilon.
\end{equation*}
\end{proof}

Next, we prove Theorem \ref{thm:uniform}

\begin{proof}
For any $A>0$, we simply replace for each fixed $i$ the point $q_i$ in Theorem \ref{thm:embedding} by points $q_i^j, j = 1, \cdots, N_i$, where $N_i = \lceil |A_i| / A\rceil$. Next, rename all the points $q_i^j$ to $p_k$. It follows that when $A = A(n,\kappa, i_0, \epsilon, t, V)$ small enough,
\begin{equation*}
1 - \epsilon < | dH_p(v) | < 1 + \epsilon.
\end{equation*}
\end{proof}

\bibliography{noisyManifold}
\bibliographystyle{plain}

\end{document}